\documentclass{amsart}
\usepackage[english]{babel}
\usepackage{amsmath}
\usepackage{amssymb}
\usepackage{amsthm}
\usepackage{latexsym}
\usepackage{tikz-cd}
\usepackage{mathtools}
\usepackage{amsfonts}
\usepackage{mathrsfs}
\usepackage{textcomp}
\usepackage[T1]{fontenc}
\usepackage{graphicx}
\usepackage{setspace}
\usepackage{nicefrac}
\usepackage{multirow}
\usepackage{indentfirst}
\usepackage{enumerate}
\usepackage{wasysym}
\usepackage{upgreek}
\usepackage{makecell}
\usepackage{hyperref}
\usepackage{booktabs}
\usepackage{etoolbox}
\usepackage{wrapfig}
\usepackage{floatflt}
\usepackage{parskip}
\usepackage{lscape}
\usetikzlibrary{automata}
\newtheorem{theorem}{Theorem}[section]
\newtheorem{lemma}[theorem]{Lemma}
\newtheorem{corollary}[theorem]{Corollary}
\theoremstyle{definition}
\newtheorem{definition}[theorem]{Definition}
\newtheorem{example}[theorem]{Example}
\newtheorem{algorithm}[theorem]{Algorithm}
\newtheorem{remark}[theorem]{Remark}
\newcommand{\mydef}[1]{{\color{blue} #1}}
\newcommand{\worry}[1]{{\bf \color{red} #1}}
\renewcommand{\worry}[1]{}
\renewcommand{\c}{{\bf c}}

\title{Necklaces count polynomial parametric osculants}
\author{Taylor Brysiewicz}
\address{Department of Mathematics, Texas A\&M University, 
College Station, TX 77840}
\email{tbrysiewicz@math.tamu.edu}
\urladdr{www.math.tamu.edu/$\sim$tbrysiewicz} 
\begin{document}
\newcommand{\ehr}{\text{ehr}}
\newcommand{\C}{\mathbb{C}}
\newcommand{\R}{\mathbb{R}}
\newcommand{\Jac}{\text{Jac}}
\newcommand{\Tor}{{\bf Tor}}
\renewcommand{\L}{\mathcal{L}}
\renewcommand{\dim}{\text{dim}}
\renewcommand{\a}{{\bf a}}
\newcommand{\la}{\langle}
\newcommand{\ra}{\rangle}
\newcommand{\V}{{\bf V}}
\newcommand{\MV}{\mathcal{MV}}
\newcommand{\vol}{{\bf vol}}
\newcommand{\Newt}{\text{Newt}}
\renewcommand{\d}{\textbf{d}}
\renewcommand{\P}{\mathbb{P}}
\newcommand{\0}{{\bf 0}}
\newcommand{\N}{\mathbb{N}}
\newcommand{\x}{{\bf x}}
\begin{abstract} 
We consider the problem of geometrically approximating a complex analytic curve in the plane by the image of a polynomial parametrization $t \mapsto (x_1(t),x_2(t))$ of bidegree $(d_1,d_2)$.  We show the number of such curves is the number of primitive necklaces on $d_1$ white beads and $d_2$ black beads. We show that this number is odd when $d_1=d_2$ is squarefree and use this to give a partial solution to a conjecture by Rababah. Our results naturally extend to a generalization regarding hypersurfaces in higher dimensions. There, the number of parametrized curves of multidegree $(d_1,\ldots,d_n)$ which optimally osculate a given hypersurface are counted by the number of primitive necklaces with $d_i$ beads of color $i$.
\end{abstract}
\maketitle
\begin{section}{Introduction}
Given a generic complex analytic curve $\mathcal C\subset \mathbb{C}^2$ through the origin defined locally by the graph of $g(x_1)=\sum_{i=1}^\infty c_i x_1^i$ it  is a common task to approximate $\mathcal C$ at the origin by a member of a simpler family of curves. The family we consider are curves which arise as the image of some polynomial map 
$$
\x(t):\mathbb{C} \to \mathbb{C}^2 $$
$$
t \mapsto (x_1(t),x_2(t))$$of bidegree $\d:=(d_1,d_2)$ and the notion of approximation we use is the degree of vanishing of the univariate power series $f(x_1,x_2)=x_2(t)-g(x_1(t))$ at $t=0$. We call $\x(t)$ a $k$-fold $\d$-parametrization when it is generically $k$-to-one. Up to reparametrization, there are finitely many $k$-fold $\d$-parametrizations which meet $\mathcal C$ to the expected maximal approximation order $d_1+d_2$ (Corollary \ref{Cor:FinitelyManySolutions}). Images of such parametrizations are curves of bidegree $\frac{\d}{k}:=\left(\frac{d_1}{k},\frac{d_2}{k}\right)$ and are called $\frac{\d}{k}$-interpolants. Because $d_1+d_2$ is the maximal approximation order attainable by a $(d_1,d_2)$-parametrization, we may assume that $f(x_1,x_2)$ is a polynomial in $\C[x_1,x_2]$ by truncating higher order terms. We show that the number of $\d$-interpolants of a generic curve is the number of primitive necklaces on $d_1$ white beads and $d_2$ black beads (Corollary \ref{Cor:PlaneResult}).

When $f \in \mathbb{R}[x_1,x_2]$, and $d_1=d_2$, Rababah conjectured that there exists at least one real $\d$-interpolant \cite{Rababah1993}. A similar conjecture was made by H\"ollig and Koch which includes the case when interpolation points are distinct and also conjectures that local approximation of a curve at a point occurs as the limit of the interpolation of distinct points on that curve \cite{HolligKoch}. The cubic case was resolved and analyzed thoroughly by DeBoor, Sabin, and H\"ollig \cite{Boor1987}. Scherer showed that there are eight $(4,4)$-interpolants and investigated bounds on the number of those which are real \cite{Scherer2003}. 
For a family of curves known as ``circle-like curves'', Rababah's conjecture has been resolved for all $d_1=d_2$ \cite{JKKZ} as well as for generic curves up to $d_1=d_2 \leq 5$  \cite{JKKZmathcomp}. By enumerating interpolants and recognizing them as solutions to polynomial systems we approach Rababah's conjecture combinatorially. We show that when $d_1=d_2$ is squarefree, a real interpolant exists for parity reasons (Theorem \ref{Thm:RealSquareFree}).
 Computations done in Section \ref{Section:Computation} provide evidence for Rababah's conjecture and also suggest that the number of real solutions has interesting lower bounds and upper bounds.

Producing a parametric description of a plane curve with particular derivatives at a point is a useful tool in Computer Aided Geometric Design particularly because these curves can achieve a much higher approximation order than Taylor approximants. For example, a quintic Taylor approximant can meet a generic curve only to order $6$ while a $(5,5)$-interpolant will meet to order $10$. Such applications do not have any preference for the behavior of the interpolating curve near infinity and so only the cases when $d_1=d_2$ have been considered. The more general problem of finding a polynomial parametrization of multidegree $\d:=(d_1,\ldots,d_n)$ osculating a hypersurface in $\mathbb{C}^n$ to approximation order $|\d|:=\sum d_i$ places the original problem into a broader theoretical context. This does not complicate the notation or proofs and so all arguments are made in the general setting. We reserve the word $\d$-interpolant for the case $n=2$ and otherwise we call these objects $\d$-osculants. 
\begin{theorem}
\label{Thm:MainTheorem}
Let $\mathcal H \subset \mathbb{C}^n$ be a generic hypersurface through $\0$. The number of $\d$-osculants of $\mathcal H$ is equal to the number of primitive $\d$-necklaces.
\end{theorem}
\begin{corollary}
\label{Cor:PlaneResult}
The number of $\d$-interpolants to a generic curve in the plane is given by the number of primitive $\d$-necklaces.
\end{corollary}
\begin{center}
{
\begin{table}[!htbp]
\caption{The number of primitive $\d$-necklaces (Sequence A$24558$ in the Online Encyclopedia of Integer Sequences \cite{OEIS}).}
\begin{tabular}{r c c c c c c c c c  }
\diaghead(1,-1){aaa}{$d_1$}{$d_2$}
& \vline & 1 & 2 & 3 & 4 & 5 & 6 & 7 &8  \\
\hline
1 & \vline & 1 & 1 & 1 & 1 & 1 & 1 & 1 & 1 \\
2 & \vline & 1 & {\bf 1} & 2 &2 &  3 & 3   & 4& 4\\
3 & \vline & 1 & 2 & {\bf 3} &5 & 7  &  9  &12 &15 \\
4 & \vline & 1 & 2 & 5 &{\bf 8} & 14 &  20  &30 &40 \\
5 & \vline & 1 & 3 & 7 &14& {\bf 25} & 42   &66 & 99\\
6 & \vline & 1 & 3 & 9 &20& 42 & {\bf 75}   & 132&212 \\
7 & \vline & 1 & 4 & 12&30& 66 &  132  &{\bf 245} & 429\\
8 & \vline & 1 & 4 & 15&40& 99 & 212   &429 &{\bf 800} 
\end{tabular}\label{fig:aperiodicTable}
\end{table}
}
\end{center}
\end{section}
\section*{Acknowledgements} I would like to thank Ulrich Reif for introducing me to this problem and for his help with existing literature. I would also like to express my gratitude to Frank Sottile for his support, thoughtful advice, and inspiring discussions. This project was supported by NSF grant DMS-1501370.
\begin{section}{Necklaces}
\label{Section:Necklaces}
Let $\d:=(d_1,\ldots,d_n) \in \mathbb{N}^n$. A \mydef{$\d$-necklace} is a circular arrangement of $d_i$ beads of color $i$ modulo cyclic rotation. A $\d$-necklace is called \mydef{$k$-fold} if it has $\frac{|\d|}{k}$ elements in its orbit under rotation and a $1$-fold necklace is called \mydef{primitive}. We denote the number of $k$-fold $\d$-necklaces by \mydef{$\mathcal N_{\d,k}$} and we let $\mydef{\mathcal{M}_{\d}}$ be the total number of $\d$-necklaces. Figure \ref{Fig:necklaces} displays the four $(3,3)$-necklaces, the first three of which are primitive while the last is $3$-fold. Figure \ref{Fig:Trinecklaces} displays the two $(1,1,1)$-necklaces, which are both primitive.
\begin{center}
\begin{figure}[!htbp]
\caption{All $(3,3)$-necklaces.}
  \begin{tikzpicture}
\draw[black] (1:1) arc (0:360:10mm);
    \node[state,fill=black,minimum size=0.5cm] at (360/6 * 1:1cm) {};
    \node[state,fill=black,minimum size=0.5cm] at (360/6 * 2:1cm) {};
    \node[state,fill=black,minimum size=0.5cm] at (360/6 * 3:1cm) {};
    \node[state,fill=white,minimum size=0.5cm] at (360/6 * 4:1cm) {};
    \node[state,fill=white,minimum size=0.5cm] at (360/6 * 5:1cm) {};
    \node[state,fill=white,minimum size=0.5cm] at (360/6 * 6:1cm) {};
   \end{tikzpicture}
  \begin{tikzpicture}
\draw[black] (1:1) arc (0:360:10mm);
    \node[state,fill=black,minimum size=0.5cm] at (360/6 * 1:1cm) {};
    \node[state,fill=black,minimum size=0.5cm] at (360/6 * 2:1cm) {};
    \node[state,fill=white,minimum size=0.5cm] at (360/6 * 3:1cm) {};
    \node[state,fill=black,minimum size=0.5cm] at (360/6 * 4:1cm) {};
    \node[state,fill=white,minimum size=0.5cm] at (360/6 * 5:1cm) {};
    \node[state,fill=white,minimum size=0.5cm] at (360/6 * 6:1cm) {};
   \end{tikzpicture}
  \begin{tikzpicture}
\draw[black] (1:1) arc (0:360:10mm);
    \node[state,fill=black,minimum size=0.5cm] at (360/6 * 1:1cm) {};
    \node[state,fill=black,minimum size=0.5cm] at (360/6 * 2:1cm) {};
    \node[state,fill=white,minimum size=0.5cm] at (360/6 * 3:1cm) {};
    \node[state,fill=white,minimum size=0.5cm] at (360/6 * 4:1cm) {};
    \node[state,fill=black,minimum size=0.5cm] at (360/6 * 5:1cm) {};
    \node[state,fill=white,minimum size=0.5cm] at (360/6 * 6:1cm) {};
   \end{tikzpicture}
   \hspace{0.1 cm}
   \vline
   \hspace{0.1 cm}
  \begin{tikzpicture}
\draw[black] (1:1) arc (0:360:10mm);
    \node[state,fill=black,minimum size=0.5cm] at (360/6 * 1:1cm) {};
    \node[state,fill=white,minimum size=0.5cm] at (360/6 * 2:1cm) {};
    \node[state,fill=black,minimum size=0.5cm] at (360/6 * 3:1cm) {};
    \node[state,fill=white,minimum size=0.5cm] at (360/6 * 4:1cm) {};
    \node[state,fill=black,minimum size=0.5cm] at (360/6 * 5:1cm) {};
    \node[state,fill=white,minimum size=0.5cm] at (360/6 * 6:1cm) {};
   \end{tikzpicture}
\label{Fig:necklaces}
\end{figure}
\end{center}
\begin{center}
\begin{figure}[!htbp]
\caption{All $(1,1,1)$-necklaces.}
  \begin{tikzpicture}
\draw[black] (1:1) arc (0:360:10mm);
    \node[state,fill=black,minimum size=0.5cm] at (360/3 * 1:1cm) {};
    \node[state,fill=yellow,minimum size=0.5cm] at (360/3 * 2:1cm) {};
    \node[state,fill=white,minimum size=0.5cm] at (360/3 * 3:1cm) {};
   \end{tikzpicture}
  \begin{tikzpicture}
\draw[black] (1:1) arc (0:360:10mm);
    \node[state,fill=black,minimum size=0.5cm] at (360/3 * 1:1cm) {};
    \node[state,fill=white,minimum size=0.5cm] at (360/3 * 2:1cm) {};
    \node[state,fill=yellow,minimum size=0.5cm] at (360/3 * 3:1cm) {};
   \end{tikzpicture}
\label{Fig:Trinecklaces}
\end{figure}
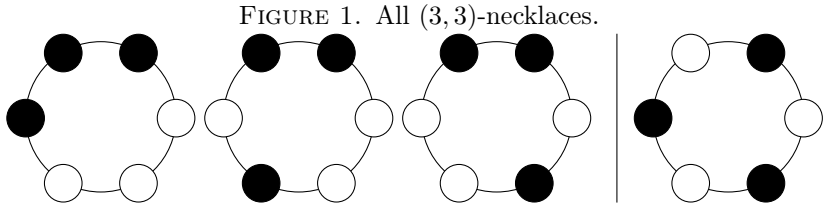
\end{center}
Observe that the number of $k$-fold $\d$-necklaces is equal to the number of primitive $\frac{\d}{k}:=\left(\frac{d_1}{k},\ldots,\frac{d_n}{k}\right)$-necklaces. This is illustrated in Figure \ref{Fig:necklaces}, where the $3$-fold necklace arises as the repetition of the only $(1,1)$-necklace, three times. This fact implies the useful formula $\mathcal{M}_{\d}=\sum\limits_{k \mid \gcd(\d)} \mathcal N_{\d,k}$.
\begin{lemma}
\label{Lemma:Recursion}
The numbers $\mathcal N_{\d,1}$ are the unique numbers satisfying the identity
$$
\binom{|\d|}{\d} =\sum_{k \mid \gcd(\d)} \frac{|\d|}{k}\mathcal N_{\frac{\d}{k},1}.$$
Where $\binom{|\d|}{\d}$ is the multinomial coefficient $\frac{|\d|!}{d_1!d_2!\cdots d_n!}$.
\end{lemma}
\begin{proof}
 Partitioning $\d$-necklaces into their orbit size gives the recursion 
\begin{align*}
\binom{|\d|}{\d} &= \sum_{k \mid \gcd(\d)} \frac{|\d|}{k}\mathcal N_{\d,k}\\
&=\sum_{k \mid \gcd(\d)} \frac{|\d|}{k}\mathcal N_{\frac{\d}{k},1}.
\end{align*}
To see that only one sequence satisfies this recursion, observe that when $\gcd(\d)=1$ the formula becomes 
$$\binom{|\d|}{\d} =|\d|\mathcal N_{\d,1}.$$
\end{proof}
\begin{remark}
The identity in Lemma \ref{Lemma:Recursion} induces the recursion 
$$\mathcal N_{\d,1}=\binom{|\d|}{\d} - \sum_{\substack{k \mid \gcd(\d) \\ k \neq 1}} \frac{|\d|}{k}\mathcal N_{\frac{\d}{k},1}$$
on the numbers $\mathcal N_{\d,1}$. 
\end{remark}
There are at least two natural actions on the set of necklaces: \mydef{reflection} and color swaps. 
A necklace can be reflected to produce another necklace. Those necklaces which are invariant under reflection are called \mydef{achiral}. A \mydef{color swap} is given by a permutation $\sigma \in \mathcal S_{n}$ where $\sigma$ acts on a necklace on $n$ colors by recoloring all beads colored $i$ instead by $\sigma(i)$. When the necklace only has two colors, color swapping is an involution whose fixed points are \mydef{self-complementary} necklaces.

 The number of necklaces on $N$ beads which are both self-complementary and achiral have been enumerated \cite{Palmer}. Let $N=2^rm$ with $m$ odd, and let $\mydef{\mathcal A_{2N}}$ be the number of self-complementary achiral necklaces on $2N$ beads. Then
\begin{equation}
\label{Eq:Achiral}\mathcal A_{2N}=\sum_{i=-1}^{r-1}2^{\lceil 2^im \rceil -1}.
\end{equation}
\begin{lemma}
\label{Lemma:AchiralEven}
The number of self-complementary achiral necklaces on $2N$ beads is even for $N >1$.
\end{lemma}
\begin{proof}
If $m>1$ then $m$ must be at least three so each summand in Equation \eqref{Eq:Achiral} is divisible by two. If $m=1$ then $r\geq 1$ and we have  
$$\mathcal A_{2^r}=\sum_{i=-1}^{r-1} 2^{\lceil 2^i \rceil -1} = 2^0+2^0+\sum_{i=1}^{r-1}2^{\lceil 2^i \rceil -1}=2+\sum_{i=1}^{r-1}2^{\lceil 2^i \rceil -1}$$ which is also even. 
\end{proof}
As mentioned in the introduction, we are primarily concerned with the case where $n=2$ and $d_1=d_2$. We investigate the parity of $\mathcal N_{(d,d),1}$.
\begin{lemma}
\label{Lemma:EvenNumberOfNecklaces}
The number of $(d,d)$-necklaces is even for all $d>2$.
\end{lemma}
\begin{proof}
The sequence $\mydef{B_{2d}}$ is the number of necklaces with $2d$ beads on two colors without any conditions on the number of beads of each color. By the color swapping involution, the parity of the number of $(d,d)$-necklaces, $\mathcal M_{(d,d)}$, is the same as the parity of $B_{2d}$. By the reflection involution, the parity of $\mathcal M_{(d,d)}$ is the same as the parity of the number of self-complementary achiral necklaces on $2d$ beads which is even by Lemma \ref{Lemma:AchiralEven}.
\end{proof}
\begin{theorem}
\label{squarefreeodd}
The number of primitive $(d,d)$-necklaces is odd if and only if $d$ is squarefree.
\end{theorem}
\begin{proof}
By Lemma \ref{Lemma:EvenNumberOfNecklaces}, the number $\mathcal M_{(d,d)}$ is even for $d>2$. 
We prove the result by induction. The result holds for $d<8$ by the diagonal of Table \ref{fig:aperiodicTable}. Suppose that $d$ is squarefree and $\mathcal N_{(k,k),1}$ is odd for all squarefree $k$ less than $d$. 
We write 
\begin{equation}
\label{Eq:Summing}
\mathcal M_{(d,d)}=\sum_{k \mid d} \mathcal N_{(k,k),1}=\mathcal N_{(d,d),1}+\left(\sum_{\overset{k \mid d}{1 \neq k \neq d}} \mathcal N_{(k,k),1} \right)+1.
\end{equation}
There are an even number of summands inside the parentheses since each divisor $k$ can be paired with the distinct divisor $\frac{d}{k}$  since $d$ is not a square. Moreover, each summand is odd by the induction hypothesis since $d$ being squarefree implies each divisor is squarefree. Thus, the sum inside the parenthesis is even. Since $\mathcal M_{(d,d)}$ is also even, Equation \eqref{Eq:Summing} implies that $\mathcal N_{(d,d),1}$ must be odd.

Conversely, suppose that $d$ is not squarefree. We will prove that $\mathcal N_{(d,d),1}$ is even by induction. Note that for $d=p^2$ with $p$ a prime, we have $\mathcal M_{(p^2,p^2)}=\mathcal N_{(p^2,p^2),1}+\mathcal N_{(p,p),1}+1$ and since $\mathcal N_{(p,p),1}$ is odd, and $\mathcal M_{(p^2,p^2)}$ is even, we see that $\mathcal N_{(p^2,p^2),1}$ must be even as well. This serves as our base of induction. 

Now, consider the case when $N$ is divisible by a square. Then we may write $$\mathcal M_{(d,d)}=\mathcal N_{(d,d),1}+\sum_{k \in K}\mathcal N_{(k,k),1}+\sum_{k \in K'}\mathcal N_{(k,k),1} +1$$
where $K$ is the set of all squarefree divisors of $d$ and $K'$ is the set of all non-squarefree divisors of $d$ other than $1$ and $d$ itself. 
We have by induction hypothesis that the summands in the right sum are even and thus do not alter the parity of $\mathcal N_{(d,d),1}$. So it is enough to prove that the number of summands in the left sum is odd. The number of squarefree divisors of $d$ is $$2^{\omega(d)}=\sum_{k \mid d}|\mu(k)|$$ where $\omega(d)$ is the number of distinct primes dividing $d$ and $\mu$ is the M\"obius function, so $$\sum_{\overset{k \mid d }{1 \neq k \neq d}} |\mu(k)|=2^{\omega(d)}-|\mu(d)|-|\mu(1)|=2^{\omega(d)}-1$$ which is odd.
\end{proof}

\end{section}
\begin{section}{Main Definitions and Results}
\label{Section:Main}
Given a hypersurface $\mathcal H \subseteq \mathbb{C}^n$, through ${\bf 0}$, we wish to find curves which osculate $\mathcal H$ optimally at $\0$. We restrict ourselves to curves which arise from $\d$-parametrizations. 
\begin{definition}
\label{Def:Dparam}
Fix $\d:=(d_1,\ldots,d_n) \in \mathbb{N}^n$.
A polynomial map $$\x(t): \mathbb{C} \to \mathbb{C}^n$$
$$t \mapsto (x_1(t),\ldots,x_n(t))$$
such that $x_i(t) \in \mathbb{C}[t]$ has degree $d_i$ and $\x(0)={\bf 0}$ is called a \mydef{$\d$-parametrization}. A $\d$-parametrization is said to be \mydef{$k$-fold} if it is generically $k$-to-one.
\end{definition}

We write a $\d$-parametrization $\x(t)$ in coordinates that describe the roots of $x_i(t)+1$, 
$$x_i(t)=\left(\prod_{j=1}^{d_i} (\alpha_{i,j}t+1)\right)-1,$$
and denote the space of $\d$-parametrizations by $\mathbb{C}_{\alpha}$. 

We define approximation order in the following algebraic way.
\begin{definition}
\label{Def:ApproxOrder}
Let $\mathcal H \subseteq \mathbb{C}^n$ be a hypersurface passing through ${\bf 0}$ given by the polynomial $$f=\sum_{I=(i_1,\ldots,i_n) \in \mathbb{N}^n} c_I x_1^{i_1}\cdots x_n^{i_n}\in \mathbb{C}[x_1,\ldots,x_n].$$ A $k$-fold $\d$-parametrization $\x(t)$ \mydef{approximates} $\mathcal H$ at ${\bf 0}$ to \mydef{order} $\gamma \in \mathbb N$ if 
\begin{equation}
\label{Equation:ApproximationOrder}
f(\x(t)) \equiv 0 \pmod {t^\gamma}.
\end{equation}
If $\gamma = |\d|:=\sum_{i=1}^n d_i$, then we say that the image $\x(\C)$ is a \mydef{$\frac{\d}{k}$-osculant} of $\mathcal H$.
\end{definition}
\begin{remark}
Because we are interested in counting the number of $\d$-osculants (geometric objects) of $\mathcal H$  rather than the number of $\d$-parametrizations approximating $\mathcal H$ to optimal order (algebraic objects), we must account for when two $\d$-parametrizations yield the same curve.
Two $\d$-parametrizations $\x(t),\hat{\x}(t)$ are said to be \mydef{reparametrizations} of one another if they have the same image. We remark that since $\x(0)=\hat{\x}(0)=0$ and both maps are $\d$-parametrizations, this implies that $\x(t)=\hat{\x}(\beta t)$ for some $\beta \in \C^*$
\end{remark}
Motivated by the definition of approximation order, we define $h_k$ to be the coefficient of $t^k$ in $f(\x(t))$. That is 
$$f(\x(t)) = \sum_{I \in \mathbb{N}^n} c_I\x(t)^I = \sum_{k=1}^\infty h_kt^k$$
so $h_k$ is regarded as a polynomial in   $\mathbb{C}[\c][\alpha]$.
Thus, the condition for $\x(t)$ to define a $\d$-parametrization meeting $\mathcal H$ to order $|\d|$ is given by the vanishing of the $|\d|-1$ polynomials $H_\d:=\{h_k\}_{k=1}^{|\d|-1}$. 
\begin{lemma}
\label{Lemma:EqHomog}
Fix $\d \in \mathbb{N}^n$. Then the polynomial $h_k\in\C[\c][\alpha]$ is bihomogeneous of degree $(1,k)$ in the $\c$ and $\alpha$ variables respectively.
\end{lemma}
\begin{proof}
Suppose that $f$ and $\x(t)$ satisfy Equation \eqref{Equation:ApproximationOrder}. The composition $f(\x(t))$ is a homogeneous linear form in the $c_I$. Note also, that the reparametrization $t \mapsto \beta t$ for some $\beta \in \mathbb{C}^*$ does not change whether or not Equation \eqref{Equation:ApproximationOrder} is satisfied, and that $t \mapsto \beta t$ is the same operation as $\alpha_{i,j} \mapsto \beta \alpha_{i,j}$. This shows that scaling the $\alpha_{i,j}$ does not change the solutions to $h_k$, so the $h_k$ are homogeneous in those variables as well. Finally, it is immediate that every factor of $t$ in the term $h_kt^k$ must come with a factor of some $\alpha$ variable and so $h_k$ is degree $k$ in the $\alpha$ variables.
\end{proof}

Lemma \ref{Lemma:EqHomog} implies that the incidence variety $V(H_\d)$ is a subvariety of the product of projective spaces $\mathbb{P}_\alpha \times \mathbb{P}_\c$. This agrees with the geometric intuition that the scaling of the equation of $\mathcal H$, or the particular parametrization of an osculant $\x(\C)$ should not change the approximation order. 

It will prove useful to embed $V(H_\d)$ into a larger projective space by keeping track of the leading coefficient of the product $\prod_{i=1}^n x_i(t)$. Namely, we write $$z^{|\d|} = \prod_{i=1}^n \prod_{j=1}^{d_i}\alpha_{i,j}.$$
This gives us the diagram 
\begin{align*}
V(H_\d) &\subseteq  \mathbb{P}_{\alpha,z} \times \mathbb{P}_\c \\
\bigg\downarrow \pi& \\
 \mathbb{P}_\c &
\end{align*}
and we are interested in the generic properties of the fibres of $\pi$.

We will begin by proving Theorem \ref{Thm:MainTheorem} for a particular parameter choice $\tilde{c}$ corresponding to the hypersurface $$\tilde{\mathcal H}:=V\left(\left(\prod_{i=1}^n (x_i+1)\right)-1\right).$$ For this fibre, Theorem \ref{Thm:MainTheorem} can be proven with an explicit bijection. We then argue that the properties of this specific fibre, such as cardinality, extend to almost all fibres.
\begin{lemma}
\label{Lemma:OurSpecialCase}
The fibre $\pi^{-1}(\tilde{\c})$ consists of $|\d|!$ simple solutions. 
\end{lemma}
\begin{proof}
The equations coming from \eqref{Equation:ApproximationOrder} can be written explicitly for $\tilde{\mathcal H}$ as
$$\left(\prod_{i=1}^n (x_i(t)+1) \right)-1\equiv 0 \pmod {t^{|\d|}}$$
$$\left(\prod_{i=1}^n \left(\prod_{j=1}^{d_i} \alpha_{i,j}t+1 \right) \right) \equiv 1 \pmod {t^{|\d|}}$$
\begin{equation}
\label{Eq:SpecialFibre}
\left(\prod_{i=1}^n \left(\prod_{j=1}^{d_i} \alpha_{i,j}t+1 \right) \right) =1 + \left(z t\right)^{|\d|}
\end{equation}
This equality induces $|\d|$ homogeneous polynomial equations to be solved in the coordinates $(\alpha,z)$ of degrees $1,2,\ldots,|\d|$.

We note that $z$ cannot be zero, since otherwise all $\alpha_{i,j}$ are zero and there are no projective solutions to $|\d|$ equations in $|\d|+1$ variables. Because of this, all solutions live in the affine open chart $z \neq 0$ and so to count the number of projective solutions, we count the number of affine solutions with $z=1$. Our condition now becomes

\begin{equation}
\label{Eq:SpecialFibreNormalized}
\left(\prod_{i=1}^n \left(\prod_{j=1}^{d_i} \alpha_{i,j}t+1 \right) \right) =1 + t^{|\d|}
\end{equation}

Note that the roots of the univariate polynomial on the right hand side of Equation \eqref{Eq:SpecialFibre} are the $|\d|$-th roots of $-1$, so the roots on the left hand side must be the same. Thus, assigning the $|d|$-th roots of $-1$ to distinct $\alpha_{i,j}$ produces all solutions to Equation \eqref{Eq:SpecialFibreNormalized}. There are exactly $|\d|!$ distinct ways to do this so there are $|\d|!$ distinct solutions to Equation \eqref{Eq:SpecialFibreNormalized}, in particular, there are finitely many solutions. B\'ezout's Theorem gives $|\d|!$ as an upper bound for the number of solutions so each solution must be simple.
\end{proof}
\begin{theorem}
\label{Theorem:Bijection}
The $\d$-osculants of $\tilde{\mathcal H}$ are in bijection with primitive $\d$-necklaces.
\end{theorem}
\begin{proof}
We have constructed all solutions to $\pi^{-1}(\tilde{\c})$ in the proof of Lemma \ref{Lemma:OurSpecialCase}. We now produce a bijection between $\d$-parametrizations coming from $\pi^{-1}(\tilde{\c})$ and circular arrangements of $d_i$ beads colored $i$. Then a bijection between images of these $\d$-parametrizations and $\d$-necklaces. Finally, a bijection between $\d$-osculants and primitive $\d$-necklaces.

First, we note that we may permute any $\alpha_{i,j}$ with $\alpha_{i,j'}$ since this only reorders the factors of $x_i(t)$.
Pick some solution $\hat \alpha$ from Lemma \ref{Lemma:OurSpecialCase}. Embed the $|\d|$-th roots of $-1$ into $\C$ and color such a point $i$ if it appears as $\hat \alpha_{i,j}$ for some $j$. Note now that this produces a bijection between $\d$-parametrizations meeting $\tilde{\mathcal H}$ to order $|\d|$ and circular arrangements of $|\d|$ roots of $-1$ (beads) with $d_i$ colored $i$.

Reparametrizing $\hat \alpha$ so that $z$ remains equal to $1$ corresponds to precomposing a parametrization with $t \mapsto \omega t$ for $\omega^{|\d|}=1$, or in other words, rotating the circular arrangement $\frac{2\pi}{|\d|}$ radians. This shows that the number of curves parametrized by our $\d$-parametrizations is equal to the number of circular arrangements of $d_i$ beads of color $i$, modulo cyclic rotation: $\d$-necklaces.

Finally, some necklaces do not parametrize a $\d$-osculant, but rather a $\frac{\d}{k}$-osculant. These parametrizations are those which appear as precompositions with $t \mapsto t^k$, or in other words, only have $\frac{|\d|}{k}$ distinct reparametrizations. Since reparametrization corresponds to cyclic rotaiton, these are the necklaces whose orbits have size $\frac{|\d|}{k}$. Therefore, the parametrizations which give $\d$-osculants are $1$-fold and are in bijection with the primitive necklaces.
\end{proof}

\begin{corollary}
\label{Cor:FinitelyManySolutions}
For generic $\hat \c \in \P_\c$, the fibre $\pi^{-1}(\hat{\c})$ is zero dimensional.
\end{corollary}
\begin{proof}
The set $C_r:=\{\hat{\c} \in \mathbb{P}_{\c} | \dim(\pi^{-1}(\hat \c))\leq r\}$ is Zariski open in $\P_\c$ (Ch 1. Sect. 6 Thm 7. \cite{shafarevich}) and so exhibiting one fibre, namely $\pi^{-1}(\tilde{c})$ whose dimension is zero implies that there is an open subset $U \subseteq \P_\c$ whose fibre dimension is zero or less. The fibre of any element in $\P_\c$ under $\pi$ is never empty because it corresponds to solving a system of $|\d|$ equations in $|\d|+1$ variables, and so every fibre of $u \in U$ has dimension $0$. 
\end{proof}
\begin{corollary}
For generic $\hat \c$, the fibre $\pi^{-1}(\hat{\c})$ consists of $|\d|!$ simple points.
\end{corollary}
\begin{proof}
By Chapter 2 Section 6 Theorem 4 of \cite{shafarevich}, the set of points of $\pi$ which have fibres of cardinality $|\d|!$ is open. Lemma \ref{Lemma:OurSpecialCase} implies that it is not empty.
\end{proof}
\newpage
{\bf Proof of Theorem \ref{Thm:MainTheorem}}

Let $\hat \c$ be a generic parameter in $\P_\c$. Then $\pi^{-1}(\hat \c)$ consists of $|\d|!$ simple points each corresponding to a $\d$-parametrization. Two $\d$-parametrizations are the same if and only if their $\alpha$ coordinates are in the same orbit of the action by $S_{d_1}\times \cdots \times S_{d_n}$ since permuting the set $\{\alpha_{i,j}\}_{j=1}^{d_i}$ leaves $x_i(t)$ fixed. So there are only $\frac{|\d|!}{d_1!d_2!\cdots d_n!}=\binom{|\d|}{\d}$ distinct parametrizations.

Partitioning these distinct parametrizations into the sets $P_k$ containing those which are $k$-fold induces the equation
$$\binom{|\d|}{\d}=\sum_{k \mid \gcd{\d}} |P_k|.$$
Note that each parametrization $\x(t) \in P_1$ has $|\d|$ reparametrizations which fix $z=1$, namely $\{\x(\omega^it)\}_{i=1}^{|\d|}$ where $\omega^{|\d|}=1$. 
However, a $\d$-parametrization in $P_k$ must appear as a $\frac{\d}{k}$-parametrization precomposed with $t\mapsto t^k$ and so there are only $\frac{|\d|}{k}$ reparametrizations fixing $z=1$. Since
$$f(\x(t^k)) \equiv 0 \pmod {t^{|\d|}} \iff f(\x(t)) \equiv 0 \pmod {t^{\frac{|\d|}{k}}}$$ we see that the images of the parametrizations in $P_k$ are the $\frac{\d}{k}$-osculants.

Letting $N_{\d}$ equal the number of $\d$-osculants, we see that $|P_k|$ is equal to $N_{\frac{\d}{k}}$ times the number of reparametrizations of a $\frac{\d}{k}$-osculant, so

$$\binom{|\d|}{\d}=\sum_{k|\gcd{(\d)}} \frac{|\d|}{k} N_{\frac{|\d|}{k}}$$
which is the necklace recurrence in Lemma \ref{Lemma:Recursion}. \hfill $\square$

One desirable property of $\d$-osculants that we have not proven is whether or not all $\d$-osculants of a regular value $\hat \c \in \mathbb{C}_\c$ are smooth. The technique of considering the hypersurface $\tilde{\mathcal H}$ and arguing that this is generic behavior is unsuccessful because not all $\d$-osculants for $\tilde{\mathcal H}$ are smooth.  Under the bijection in Theorem \ref{Theorem:Bijection}, a singular $\d$-osculant corresponds to a primitive $\d$-necklace embedded into $\C$ via $|\d|$-th roots of $-1$ such that the sum of each subset of $|\d|$-th roots of $-1$ colored $i$ is zero: these are the linear terms of the $x_i(t)$ and $\x(\mathbb{C})$ is singular whenever the linear terms of each $x_i(t)$ are zero. There exists a primitive $(9,9)$-necklace with this property, and thus there exists a singular $(9,9)$-interpolant of $\tilde{\mathcal H}$. This necklace is depicted in Figure \ref{Fig:BalancedNecklace}.
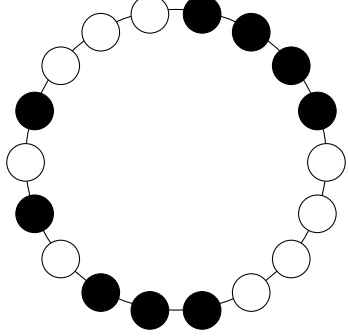
\begin{figure}[!htbp]
\caption{A primitive $(9,9)$-necklace corresponding to a singular $(9,9)$-interpolant of $\tilde{\mathcal H}$.}
  \begin{tikzpicture}
\draw[black] (1:2) arc (0:360:20mm);
    \node[state,fill=black,minimum size=0.5cm] at (360/18 * 1:2cm) {};
    \node[state,fill=black,minimum size=0.5cm] at (360/18 * 2:2cm) {};
    \node[state,fill=black,minimum size=0.5cm] at (360/18 * 3:2cm) {};
    
    \node[state,fill=black,minimum size=0.5cm] at (360/18 * 4:2cm) {};
    \node[state,fill=white,minimum size=0.5cm] at (360/18 * 5:2cm) {};
    \node[state,fill=white,minimum size=0.5cm] at (360/18 * 6:2cm) {};
    \node[state,fill=white,minimum size=0.5cm] at (360/18 * 7:2cm) {};
    \node[state,fill=black,minimum size=0.5cm] at (360/18 * 8:2cm) {};
    \node[state,fill=white,minimum size=0.5cm] at (360/18 * 9:2cm) {};
    \node[state,fill=black,minimum size=0.5cm] at (360/18 * 10:2cm) {};
    \node[state,fill=white,minimum size=0.5cm] at (360/18 * 11:2cm) {};
    \node[state,fill=black,minimum size=0.5cm] at (360/18 * 12:2cm) {};
    \node[state,fill=black,minimum size=0.5cm] at (360/18 * 13:2cm) {};
    \node[state,fill=black,minimum size=0.5cm] at (360/18 * 14:2cm) {};
    \node[state,fill=white,minimum size=0.5cm] at (360/18 * 15:2cm) {};
    \node[state,fill=white,minimum size=0.5cm] at (360/18 * 16:2cm) {};
    \node[state,fill=white,minimum size=0.5cm] at (360/18 * 17:2cm) {};
    \node[state,fill=white,minimum size=0.5 cm] at (360/18 * 18:2cm) {};

   \end{tikzpicture}
   \end{figure}
   \label{Fig:BalancedNecklace}
\end{section}
\begin{section}{Real solutions}
\label{Section:Real}
If there is an odd number of solutions to a polynomial system defined over $\R$, then there must be at least one real solution.
However, our solution count in Theorem \ref{Thm:MainTheorem} is not \emph{a priori} the count of a system of real polynomials. In fact, the only such solution count we have determined is that of $H_\d \cup \{z=1\}$ which has $|\d|!$ solutions. Therefore, this does not directly imply that when $\mathcal N_{\d}$ is odd a real solution must exist. However, we show that this does happen to be the case.
\begin{lemma}
\label{Lemma:ComeInConjugates}
The $k$-fold solutions of $\pi^{-1}(\hat \c)$ for a parameter $\hat \c \in \R_\c$ are fixed as a set under complex conjugation.
\end{lemma}
\begin{proof}
We prove this using induction on the number of divisors of $\gcd{(\d)}$. For the base case, suppose that $\gcd{(\d)}=1$. Then all solutions must correspond to $1$-fold $\d$-parametrizations. Thus, the $|\d|!$ solutions are fixed under conjugation as they are solutions to a real system of polynomial equations.

Suppose now that $\gcd{(\d)}$ has proper divisors. The solutions to $H_{\d}\cup \{z=1\}$ are fixed under conjugation as a set because they are the solutions to a real polynomial system. Moreover, the set of $\d$-parametrizations which are not $1$-fold is fixed under conjugation by induction hypothesis, so the solutions left over (namely the $1$-fold $\d$-parametrizations) must be as well. 
\end{proof}
\begin{lemma}
\label{Lemma:OddImpliesReal}
If $\mathcal N_{\d}$ is odd, there is at least one real $\d$-osculant.
\end{lemma}
\begin{proof}
Partitioning the $|\d|!$ solutions to $H_\d \cup \{z=1\}$ into sets determined by whether or not they are $k$-fold gives the recursion
$$|\d|! = \sum_{k | \gcd{\d}} \left(\prod_{i=1}^n d_i! \right) \frac{|\d|}{k} \mathcal N_{\d/k}$$
By Lemma \ref{Lemma:ComeInConjugates}, we know that the $\left(\prod_{i=1}^{n}d_i! \right) |\d| \mathcal N_{\d}$ $1$-fold $\d$-parametrizations are fixed under complex conjugation as a set.

Recall that the factor of $\left(\prod_{i=1}^{n}d_i! \right) |\d|$ occurs because for each $\d$-osculant, there are $|\d|$ reparametrizations, and $\prod_{i=1}^n d_i!$ ways to relabel the roots of $x_i(t)$. Let $p_1,\ldots,p_{\mathcal N_\d}$ be all $\d$-osculants and let $S_i$ denote the class of  all solutions which correspond to $p_i$. 

Let $\alpha \in S_i$ and consider a relabeling of the roots via $(\sigma_1,\ldots,\sigma_n) \in S_{d_1} \times \cdots \times S_{d_n}$ so that $\alpha_{i,j} \mapsto \alpha_{i,\sigma_i(j)}$. If such a relableing induces complex conjugation (if  $\sigma(\alpha)= \overline{\alpha}$) then the roots of all $x_i(t)$ are fixed under conjugation and thus $\alpha$ corresponds to a real $\d$-osculant. If any reparametrization $\alpha \mapsto \omega^k\alpha$ for $\omega^{|\d|}=1$ induces complex conjugation, then 
\begin{align*}
\omega^k \alpha &= \overline{\alpha} \\
\implies \overline{ \omega^{k/2} \alpha} = \omega^{-k/2} \overline{ \alpha} &= \omega^{-k/2}\omega^{k} \alpha = \omega^{k/2}\alpha
\end{align*} 
and so the reparametrization $ \omega^{k/2}\alpha$ is fixed under complex conjugation, so $p_i$ must be real.

Therefore, if $\mathcal N_\d$ is odd, then either (1) there is a real solution, or (2) none of the $\left(\prod_{i=1}^n d_i!\right) |\d|$ solutions corresponding to $p_i$ are conjugates of one another. Therefore, the classes $S_i$ must be conjugates of eachother set-wise. But that is a contradiction, since there are an odd number of classes.
\end{proof}

\begin{theorem}
\label{Thm:RealSquareFree}
Let $\mathcal C$ be a generic curve in the plane defined by a real polynomial. For any squarefree integer $d$ there exists at least one real $(d,d)$-interpolant. 
\end{theorem}
\begin{proof}
This follows directly from Lemma \ref{Lemma:OddImpliesReal} and Theorem \ref{squarefreeodd}.
\end{proof}
\end{section}
\begin{section}{Computations}
\label{Section:Computation}
\mydef{Homotopy continuation}, a tool in numerical algebraic geometry, provides an extremely quick way to produce solutions to a particular polynomial system when solutions to a similar system have been precomputed. Briefly, the method constructs a homotopy from the polynomial system whose solutions are known (called the start system) to the target polynomial system, whose solutions are desired. Then the start solutions are tracked using predictor-corrector methods toward the target solutions. 

Since we have an explicit description of all $\d$-osculants of $\tilde{H}$ given by necklaces, this method is perfectly suited for the problem of computing $\d$-osculants for a generic $\mathcal H$. We outline the process in Algorithm \ref{Alg:homotopycontinuation}.

{
\begin{algorithm}(Finding all $\d$-osculants)

\fbox{

\label{Alg:homotopycontinuation}
\begin{tabular}{l}
{\bf Input:} $\d \in \mathbb{N}^n$, $\mathcal I \subseteq \mathbb{N}^n, \{c_I\}_{I \in \mathcal I} \subseteq \mathbb{C}$\\
{\bf Output:} All $\d$-osculants of $V(f)$ where $f=\sum\limits_{I \in \mathcal I} c_Ix^I$. 
\rule{0 pt}{0.3 cm} \\
\hline 
\hline
\rule{0 pt}{0.7 cm}
1) Compute all primitive $\d$-necklaces via set partitions of $\{1,\ldots,|\d|\}$ \\
so that $b_{i,j}$ is the $j$-th element of the $i$-th part of the set partition.\\

2) For each primitive $\d$-necklace, compute $\x(t)=(x_1(t),\ldots,x_n(t))$ where \\

$x_i(t) = (-1)+\prod_{j=1}^{d_i} (\omega^{b_{i,j}}t+1).$\\

3) Set the starting parameters to be those coming from $\tilde{\mathcal H}$. \\

4) Set the starting points to be the solutions computed in Step 2.\\

5) Track the solutions of the equations given by \eqref{Equation:ApproximationOrder} by \\ varying the parameters $c_I$ from  those corresponding to $\tilde{\mathcal H}$ towards those \\corresponding to $f$.\\

7) Return the solutions given by the homotopy.

\end{tabular}}
\end{algorithm}
}

The software {\bf alphaCertified} can certify if a solution is real or not and relies on Smale's $\alpha$-theory \cite{alphaCertify, smale}.
Algorithm \ref{Alg:homotopycontinuation} computes $\d$-osculants in the $\alpha$ variables and so solutions which correspond to real osculants probably do not have real coordinates. Therefore, to certify that solutions are real, we expand the expressions $$x_i(t) = -1+\prod_{j=1}^{d_i} (\alpha_{i,j} t+1)=\sum_{j=1}^{d_i} a_{i,j}t^i$$ and normalize so that $a_{1,1}=1$. Even though we have not proven that $a_{1,1}$ is generically nonzero, we have only seen this behavior in the computational experiments. After this normalization, real interpolants do correspond to solutions with real coordinates and we can use {\bf alphaCertified} to certify the number of real solutions. 

We implemented Algorithm \ref{Alg:homotopycontinuation} in {\bf Macaulay 2} using the {\bf Bertini.m2} package to call the numerical software {\bf Bertini} for the homotopy continuation \cite{BHSW06,M2,BertiniSource}. Using the implementation, we computed many instances of the problem of finding $(d_1,d_2)$-interpolants and we tally the number of real solutions for the problems in Table \ref{Tab:Real} where the row labeled $k$ indicates that $\mathcal{N}_{(d_1,d_2)}\pmod{2}+2k$ real solutions were found. The current certified results can be found on the author's webpage \cite{taylor}.

As one can see from Table \ref{Tab:Real} that when $d_1=d_2$ it seems that Rababah's conjecture holds. Moreover, in the case of $(4,4)$ and $(5,5)$ there seem to be nontrivial upper bounds to the number of real solutions, namely $6$ and $15$ respectively.

\begin{table}[!htbp]
\begin{center}
\resizebox{\columnwidth}{!}{%
\begin{tabular}{c r c c c c c c c c c c}
({$d_1$},{$d_2$})  &\vline & (2,3) & (2,4) & (2,5) & (3,3) & (3,4) & (3,5) & (4,4) & (4,5) & (5,5) \\
\hline{$\mathcal N_{(d_1,d_2)}$} &\vline & 2 & 2 & 3 & 3 & 5 & 7 &  8 & 14 & 25&  \\
\hline
&\vline &&&&&&&&&
\\
{\multirow{3}{*}{\rotatebox[origin=c]{90}{$\text{row}=\frac{\#\text{real sols} -\mathcal{N}_{(d_1,d_2)}\pmod{2}}{2}$}}}&0  \vline & 84247& 102629 & 195490&414314 & 414314& 1925&0 & 73&0 &  \\
 &\vline &&&&&&&&&
\\
&1 \vline & 486533& 432605& 313559 &39985 & 405142 & 300265&125841 &6344 &138 &  \\
&\vline &&&&&&&&&
\\
&2 \vline &       &  & - & - &30358 &71383 &336261  &38692 &15795&  \\
&\vline &&&&&&&&&
\\
&3 \vline &- & -&  &  &  & 12072& 62&102139 &16309 &  \\
&\vline &&&&&&&&&
\\
& 4 \vline &- & -&- &- & - &- &0 &15517  &3182&  \\
&\vline &&&&&&&&&
\\
&5 \vline &- & -&- &- &  & & -& 19&102 &  \\
&\vline &&&&&&&&&
\\
&6 \vline &- & -&- &- & -& -& & 0 &3&  \\
&\vline &&&&&&&&&
\\
&7 \vline &- & -&- &- & -& &- & -& 5&  \\
\end{tabular}
}
\end{center}
\caption{Results of Computational Experiments}
\label{Tab:Real}
\end{table}

\begin{example}[A curve with six real $(4,4)$-interpolants]
\label{Ex:Sixreal}
Consider the curve defined by \begin{align*}
f(x,y)=&\left(-586971\right)x+\left(-481753\right)x^2+\left(114414\right)x^3+\left(-361929\right)x^4+\\
&\left(152011\right)x^5+\left(-616310\right)x^6+\left(244262\right)x^7-1000000y.
\end{align*}
The eight $(4,4)$-interpolants are given (approximately) by
{
\begin{align*}
s_1:x(t)&=(.166+1.601i)t^4+(.028-.204i)t^3+(-.113-1.053i)t^2+t,\\y(t)&= (.003-1.219i)t^4+(.207+1.134i)t^3+(-.415+.618i)t^2-.587t\\
s_2:x(t)&=      (.166-1.601i)t^4+(.028+.204i)t^3+(-.113+1.053i)t^2+t, \\
y(t)&= (.003+1.219i)t^4+(.207-1.134i)t^3+(-.415-.618i)t^2-.587t\\
s_3:x(t)&= .031t^4-.537t^3-.065t^2+t,\\
y(t)&= .113t^4+.492t^3-.444t^2-.587t\\
s_4:x(t)&= -9.902t^4+4.516t^3+2.234t^2+t,\\
y(t)&= -.538t^4-4.689t^3-1.793t^2-.587t\\
s_5:x(t)&=-.347t^4-.787t^3+.388t^2+t,\\
y(t)&=       .661t^4+.203t^3-.709t^2-.587t\\
s_6:x(t)&= 8.902t^4+2.333t^3-1.772t^2+t, \\
y(t)&=     -9.956t^4+.452t^3+.558t^2-.587t\\
 s_7:x(t)&=.162t^4-.799t^3-.349t^2+t,\\
 y(t)&= .134t^4+.92t^3-.277t^2-.587t\\
s_8:x(t)&= -.613t^4-2.228t^3+.031t^2+t,\\
y(t)&= 2.155t^4+1.392t^3-.5t^2-.587t.\\
\end{align*} 
}

\begin{figure}
\caption{Six real $(4,4)$-interpolants}
\includegraphics[scale=0.5]{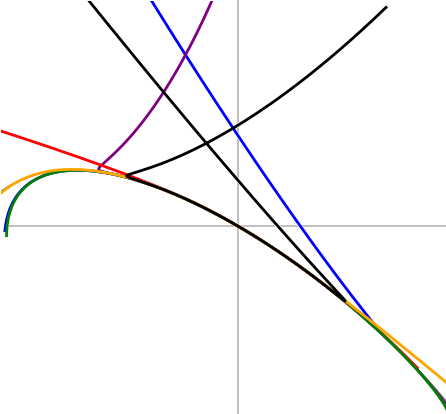}
\label{SixReal}
\end{figure}
Here, $s_3,\ldots,s_8$ define real curves. Figure \ref{SixReal} plots their branches near $t=0$.

\end{example}
\end{section}
\newpage
\bibliographystyle{plain}
\bibliography{ref}

\end{document}